\documentclass[reqno,12pt]{amsart}
\usepackage{amsthm}
\usepackage{amssymb,amsmath}
\usepackage{xypic}
\usepackage{hyperref}
\usepackage{bbm}

\title{Symplectic Dolbeault Operators on K\"ahler Manifolds}
\author{Eric O. Korman}
\email{ekorman@math.upenn.edu}
\address{Department of Mathematics, University of Pennsylvania, Philadelphia, PA 19104}
\date{}

\newcommand{\C}{\mathbb{C}}

\newcommand{\D}{\mathscr{D}}

\newcommand{\R}{\mathbb{R}}

\newcommand{\Z}{\mathbb{Z}}

\newcommand{\g}{\mathfrak g}

\newcommand{\lspan}{\operatorname{span}}
\newcommand{\End}{\operatorname{End}}

\newcommand{\om}{\omega}
\newcommand{\Om}{\Omega}
\newcommand{\ch}{\text{ch}}

\newcommand{\Hom}{\text{Hom}}
\newtheorem{lem}{Lemma}[section] 

\newtheorem{cor}{Corollary}[section]

\newtheorem{defin}{Definition}[section]
\newtheorem{prop}{Proposition}[section]
\newtheorem{defprop}{Definition/Proposition}[section]
\newtheorem*{rk}{Remark}
\newtheorem*{ack}{Acknowledgements}

\renewcommand{\t}{\mathfrak t}
\renewcommand{\sp}{\mathfrak{sp}}

\numberwithin{equation}{section}

\usepackage{slashed}
\usepackage{mathrsfs}

\begin{document}
\maketitle

\begin{abstract}
For a K\"ahler Manifold $M$, the ``symplectic Dolbeault operators" are defined using the symplectic spinors and associated Dirac operators, in complete analogy to how the usual Dolbeault operators, $\bar\partial$ and  $\bar\partial^*$, arise from Dirac operators on the canonical complex spinors on $M$.  
We give special attention to two special classes of K\"ahler manifolds: Riemann surfaces and flag manifolds ($G/T$ for $G$ a simply-connected compact semisimple Lie group and $T$ a maximal torus).  In the case of flag manifolds, we work with the Hermitian structure induced by the Killing form and a choice of positive roots (this is actually not a K\"ahler structure but is a K\"ahler with torsion (KT) structure).  For Riemann surfaces the symplectic Dolbeault operators are elliptic and we compute their indices.  In the case of flag manifolds, we will see that the representation theory of $G$ plays a role and that these operators can be used to distinguish (as Hermitian manifolds) between the flag manifolds corresponding to the Lie algebras $B_n$ and $C_n$.  We give a thorough analysis of these operators on $\C P^1$ (the intersection of these classes of spaces), where the symplectic Dolbeault operators have an especially interesting structure.
\end{abstract}

\section{Introduction}
Symplectic spin geometry is the analogue of the usual (orthogonal) spin geometry, with the Weyl algebra bundle replacing the Clifford algebra bundle and the metaplectic representation of the symplectic Lie algebra replacing the spin representation of the orthogonal Lie algebra.  One of the difficulties of symplectic spin geometry is that, since the metaplectic representation is infinite dimensional, the bundle of spinors is now an infinite rank bundle.  However, if the manifold has a compatible complex structure then this bundle naturally splits into a direct sum of finite dimensional subbundles (this is because the metaplectic representation splits into a direct sum of irreducible subrepresentations when restricted to $\mathfrak u(n) \subset \mathfrak{sp}(n)$).  Another issue is the fact that the symplectic Dirac operators are not elliptic, so that the usual approaches of index theory do not work.  \\

In this paper we continue the study of these operators, which was initiated by K. Habermann in \cite{haberk,haber}.  The central players in our work are the ``symplectic Dolbeault operators," which are naturally arising combinations of the Dirac operators.  Our main results are the indices of these operators in the case of Riemann surfaces (the only such manifolds for which the operators are elliptic), a representation theoretic formula for the ``ground state" spectrum of an associated second order operator in the case of flag manifolds (which we use to differentiate the flag manifolds $B_n$ and $C_n$), and a thorough analysis for $\C P^1$. \\

The organization of the paper is as follows.  In section 2 we first provide some background and motivation by recalling the basic facts of the orthogonal spin geometry of a K\"ahler manifold.  This will be similar to the approach in \cite{mich}.  Next we give the basics of the symplectic spin geometry of a K\"ahler manifold, following very closely the account in \cite{haber}. It is here that we define the symplectic Dolbeault operators, $\D$ and $\bar\D$. \\

In section 3 we specialize to Riemann surfaces, where all of the symplectic Dolbeault operators are elliptic.  A straightforward application of the Atiyah-Singer index theorem computes their indices.  \\

In section 4 we then specialize to the case of flag manifolds $G/T$, which are studied using the representation theory of $G$.  For $E_0$ the ``ground state" bundle of the harmonic oscillator $H$, $\ker \bar\D\vert_{E_0}$ is just holomorphic sections of $E_0$, and can therefore be computed using the Borel-Weil theorem.  We also compute the spectrum of the elliptic operator $\frac{1}{2} [\D, \bar \D]$ explicitly in terms of infinitesimal characters of the Casimir element of $G$.  A consequence of this result is a quick proof that $Spin(2n+1)/ T$ and $Sp(n)/T$ are not isomorphic as hermitian manifolds for $n \ge 3$. We note that the Hermitian structure we use on $G/T$ uses the Killing form as metric and so is not a K\"ahler structure (unless $G$ is a product of $SU(2)$'s), but a K\"ahler with torsion structure.  In an earlier version of this paper, we made the erroneous claim that this makes $G/T$ into a K\"ahler manifold.  In forthcoming work, we analyze the symplectic Dolbeault operators for K\"ahler structures on $G/T$. \\

Finally, in section 5 we analyze the case of $\C P^1 = SU(2) / U(1)$, where $P$ has an especially nice form.  Here we are able to explicitly diagonalize the section of the spinor bundles with respect to $H$ and $\frac{1}{2}[\D,\bar\D]$.  By what seems to be a computational coincidence, we are then able to decompose the sections of the spinor bundle into finite dimensional representations of the algebra of differential operators generated by the Dirac operators and $H$.

\begin{ack}
The author is very grateful to Jonathan Block for his guidance and numerous helpful discussions.
\end{ack}

\section{Background}

\subsection{Orthogonal spin geometry of K\"ahler manifolds}\label{orthspin}
We now give a brief account of the spin geometry of a K\"ahler manifold $(M^{2n},J,\om, g)$.  A good reference for more details is \cite{lawmich}.  Canonically associated to the complex structure on $M$ is the bundle of complex spinors given by anti-holomorphic forms.  The Clifford action is given by
\[
v \cdot \mu = \sqrt{2}(v^{0,1} \rfloor \mu + (v^{1,0})^\flat \wedge \mu)
\]
where
\[
v = v^{0,1} + v^{1,0} \in TM, ~~ v^{0,1} \in T^{1,0} M, ~ v^{1,0} \in T^{0,1} M, ~~\mu \in \Omega^{0,*}(M),
\]
and $\flat$ denotes the isomorphism $TM \otimes \C \to T^*M \otimes \C$ induced by $g$.  Then $v \cdot v \cdot$ is multiplication by the scalar $-g(v,v)$ and so this action makes $S := \Omega^{0,*}(M)$ a bundle of representations of $Cl(M)$, the bundle of Clifford algebras.  One can then define the associated Dirac operator
\[
D: \Gamma(S) \stackrel{\nabla}{\to} \Gamma(S \otimes T^*M) \stackrel{g}{\to} \Gamma(S \otimes T M) \to \Gamma(S)
\]
where $\nabla$ is the connection on $\Lambda^* (T^{1,0})^* M$ induced by the Levi-Civita connection on $M$ and the last map is the Clifford action.  One can obtain a second Dirac operator, $\tilde D$, by using $\om$ to lower indices instead of $g$ in the second map.  These two Dirac operators are the operators appearing in \cite{mich}, which are $D$ and $D^C$, respectively.  Then it turns out \cite{mich} that 
\[
D = \sqrt 2(\bar\partial + \bar\partial^*), ~~ \tilde D = -i\sqrt 2(\bar\partial - \bar\partial^*).
\]
One thus recovers $\bar\partial$ and $\bar\partial^*$ as 
\begin{equation}\label{orthdol}
\bar\partial = \frac{1}{\sqrt 2}(D + i\tilde D), ~~ \bar\partial^* = \frac{1}{\sqrt 2}(D - i\tilde D).
\end{equation}

Though Clifford modules, such as $S$, are $\Z/2$ graded, we can actually recover the $\Z$ grading on $S$ as follows.  We first recall that the second degree filtration of $Cl(n)$ of the Clifford algebra on $\R^n$ is isomorphic, as a Lie algebra, to $\mathfrak o(n)$ (the isomorphism being $u \mapsto (v\mapsto [u,v])$, the commutator taking place in the Clifford algebra).  We can thus view the bundle $Cl^{(2)}(M)$ as a Lie subalgebra bundle of $\End(TM)$.  Thus the complex structure $J$, which is skew-symmetric, determines a section of $Cl^{(2)}(M)$, which we will denote by $H$.  One can see that
\[
H = \frac{1}{2} \sum_{j=1}^n e_j Je_j
\]
where $\{e_1, Je_1, \ldots, e_n, Je_n\}$ is a (local) orthonormal frame for $TM$.  It is then straightforward to verify
\begin{prop}
The eigenspace decomposition of $H$ acting on $S$ recovers the $\Z$ grading on $S$.  More specifically, $\Lambda^{0, l} T^*M$ is the eigenspace of $H$ with eigenvalue $(l - \frac{n}{2})i$.
\end{prop}
\begin{cor}
We have the following commutation relations
\[
[H, \bar\partial] = \bar\partial, ~~ [H,\bar\partial^*] = - \bar\partial^*.
\]
\end{cor}

\subsection{The Weyl algebra and metaplectic representation}
Throughout let $(V^{2n}, \om_0, J_0, g_0)$ be a Hermitian vector space.  We use the same conventions as \cite{haber} and often use a subscript or superscript zero for objects that will induce global geometric objects in subsequent sections.  The Weyl algebra, $W(V)$, is the free complex algebra generated by $V$ subject to
\[
vw - wv = -\om_0(v,w), ~~ v,w\in V.
\]
We call $\{a_1,b_1,\ldots, a_n, b_n\}$ a symplectic basis if
\[
\om_0(a_i, b_j) = \delta_{ij}, ~~ \om_0(a_i,a_j) = 0 = \om_0(b_i,b_j).
\]
The analog of the spin representation of the Clifford algebra is now canonical quantization, i.e. the representation of $W(V)$ on $L^2(\R^n)$ (by unbounded operators) determined by
\[
\sigma: a_i^k \mapsto (ix_j)^k, ~~ b_i^k \mapsto \frac{\partial^k}{\partial x_j^k},1 \mapsto i.
\]
where $x_1,\ldots, x_n$ are coordinates on $\R^n$.  Note that every element of $A$ acts anti-self-adjointly.  We also remark that $\sigma$ is not quite a true representation as an algebra since $\sigma$ is not an algebra homomorphism.  However, we have
\begin{equation}\label{com}
[\sigma(v), \sigma(w)]  = -i\om_0(v,w), ~~ v,w\in V.
\end{equation}
Similar to the orthogonal case, the subspace of $W(V)$ given by quadratic elements, which we denote by $W^{(2)}(V)$, is isomorphic as a Lie algebra to $\mathfrak{sp}(V)$.  The isomorphism is given by
\[
W^{(2)}(V) \ni u \mapsto (v \mapsto [u,v]), ~~ v \in V,
\]
where the commutator is taking place in $W(V)$ and we view $V$ inside of $W(V)$ as the degree 1 part of the filtration. \\

Denote by $m_*$ the metaplectic representation of $\mathfrak{sp}(V)$.  This is related to $\sigma$ by the equation
\[
m_* = -i\sigma \vert_{\mathfrak{sp}(V)}.
\]
Actually, for our purposes it is sufficient to take this as a definition of $m_*$.  The representation $m_*$ does not integrate to a representation of $Sp(V)$ but does to its double cover $Mp(V)$, the metaplectic group. \\

We call a symplectic basis $\{a_j, b_j\}$ \textit{unitary} if $b_j = Ja_j$.  The center of $\mathfrak u(V) \subset \mathfrak{sp}(V)$ is spanned by $J_0$ and under the isomorphism $\mathfrak{sp}(V)\simeq W^{(2)}(V)$, $J_0$ corresponds to
\[
\frac{1}{2}\sum_{j=1}^n (a_j^2 + b_j^2),
\]
where $\{a_j,b_j \}$ a unitary basis.  We define
\[
H_0 := \sigma(J_0) = \frac{1}{2} \sum_{j=1}^n \left(-x_j^2 + \frac{\partial^2}{\partial x_j^2}\right),
\]
which is the Hamiltonian for the quantum harmonic oscillator. \\

Since $J_0$ lies in the center of $\mathfrak u(V)$, the restriction of the metaplectic representation splits into a direct sum of eigenspaces of $H_0$.  As is well-known, the spectrum of $H_0$ is $\{-(l+n/2): l = 0, 1, \ldots\}$ and the eigenspace $E_l^0$ with value $-(l+n/2)$ is spanned by the functions
\[
\{h_{\beta_1}(x_1) \cdots h_{\beta_n}(x_n) : \beta_j \in \Z_{\ge 0}, ~~ \beta_1 + \ldots + \beta_n = l\}
\]
where
\[
h_m(t) = e^{t^2/2} \frac{d^m}{dt^m} e^{-t^2}
\]
are the Hermite functions.  Thus $\dim E_l^0 = {n+l-1 \choose l}$.  We recall the following useful equations, familiar from the standard treatment of the quantum harmonic oscillator:
\begin{gather*}
(t- \frac{d}{dt}) h_m(t) = -h_{m+1}(t), ~~~ (t+\frac{d}{dt}) h_m(t) = -2m h_{m-1}(t) \\
\frac{1}{2}(t^2 - \frac{d^2}{dt^2}) h_m = (m+1/2) h_m
\end{gather*}
\begin{equation}\label{hermiteinner}
\langle h_j, h_k \rangle = \sqrt{\pi} 2^j \delta_{jk}
\end{equation}

Thus if $\{a_j,b_j\}$ is a unitary basis and $Z_j := \frac{1}{2}(a_j - i b_j), \bar Z_j := \frac{1}{2}(a_j + i b_j)$, then the operators $\sigma(Z_j)$ and $\sigma(\bar Z_j)$ are raising and lowering operators in the $j$th direction, respectively:
\begin{equation}\label{ladderops}
\begin{aligned}
&\sigma(Z_j) : E_l \to E_{l+1}, ~~ h_{\beta_1\cdots \beta_n} \mapsto -\frac{i}{2} h_{\beta_1 \cdots \beta_{j-1} \beta_j + 1 \beta_{j+1} \cdots \beta_n} \\
&\sigma(\bar Z_j) : E_l \to E_{l-1}, ~~ h_{\beta_1\cdots \beta_n} \mapsto -i\beta_j h_{\beta_1 \cdots \beta_{j-1} \beta_j - 1 \beta_{j+1} \cdots \beta_n} 
\end{aligned}
\end{equation}
\begin{equation*}
\sigma(Z_j)\sigma(\bar Z_j) + \sigma(\bar Z_j)\sigma(Z_j) : E_l \to E_l, ~~~ h_{\beta_1\cdots \beta_n} \mapsto -(\beta_j + 1/2) h_{\beta_1\cdots \beta_n}
\end{equation*}
where we write $h_{\beta_1\cdots\beta_n}$ for $h_{\beta_1}(x_1)\cdots h_{\beta_n}(x_n)$. \\

We may take $\{\frac{1}{2}(a_j^2 + b_j^2) = Z_j \bar Z_j + \bar Z_j Z_j\}$ as a basis for a maximal torus $\t$ of $\mathfrak u(V) \subset W(V)$.  Note that $Z_j \bar Z_j + \bar Z_j Z_j$ corresponds to $\text{diag}(0,\cdots,0, i, 0 ,\cdots 0) \in \mathfrak u(n)$ under the isomorphism $W^{(2)}(V) \simeq \mathfrak u(n)$.  By the above formula and the fact that the metaplectic representation is given by $-i\sigma$, we see that $h_{\beta_1 \cdots \beta_n}$ is a weight vector for $E_l^0$ with weight 
\begin{equation}\label{weights}
\nu_{\beta_1 \cdots \beta_n} := \sum_{j=1}^n (\beta_j + \frac{1}{2}) i e_j = \sum_{j=1}^n \beta_j i e_j + \nu_{0\cdots 0},
\end{equation}
where $\{e_j\}$ is the basis of $\t^*$ dual to $\{\frac{1}{2}(a_j^2 + b_j^2)\}$.  It follows that $E^0_l$ does not lift to a representation of $U(V)$ (but of course is a representation of its double cover $\tilde U(V) \subset Mp(V)$).  Each representation $E_l^0$ is irreducible \cite{borwal} and from the equation for the weights, we see that
\begin{equation} \label{fockmeta}
\begin{aligned}
E_0^0 \otimes E_0^0 = \Lambda^n \C^n \\
E^0_l \simeq  S^l \C^n \otimes E_0^0.
\end{aligned}
\end{equation}

\subsubsection{The Fock or algebraic metaplectic representation}\label{focksection}
There is an alternative method of constructing the metaplectic representation of the Weyl algebra.  Implicit in the construction of the previous section was a choice of complementary Lagrangian subspaces ($L:=\lspan \{a_j\}$ and $JL = \lspan\{b_j\}$).  We can mimmick this construction by using complementary maximal isotropic subspaces for the compatible inner product $g_0$ on $V \otimes \C$, instead of $\om_0$.  This is useful because it is globalizable to any K\"ahler manifold: the holomorphic and anti-holomorphic tangent bundles always form such a pair of maximally isotropic subbundles of $TM\otimes \C$ (in contrast, it is rare to be able to find a Lagrangian polarization). \\

To construct this representation, let $V^{1,0}$ be the $i$ eigenspace of $J_0$ (extended to $V\otimes \C$) and let $V^{0,1}$ be the $-i$ eigensapce.  For a unitary basis $\{a_j,b_j\}$ we define
\[
Z_j := \frac{1}{2}(a_j - ib_j), ~~ \bar Z_j := \frac{1}{2}(a_j + ib_j).
\]
Then we take the representation space to be $S^* V^{1,0} = \C[Z_1,\ldots, Z_n]$, with the action of $W(V)$ being determined by $Z_j \mapsto -\frac{1}{2} i Z_j, ~~ \bar Z_j \mapsto -i \frac{\partial}{\partial Z_j}$.  The usual grading on $S^* V^{1,0}$ is recovered by the eigenspaces of the action of $J \in W^{(2)}(V)$.  The isomorphism to the metaplectic representation described in the previous section is given by
\[
Z_1^{\beta_1} \cdots Z_n^{\beta_n} \mapsto h_{\beta_1 \cdots \beta_n}.
\]
Thus transferring over the $L^2$ inner product induces a hermitian inner product on $S^* V^{1,0}$ that makes the monomials $\{Z_1^{\beta_1} \cdots Z_n^{\beta_n}\}$ an orthogonal basis.  By scaling the inner product by $\frac{1}{2\sqrt\pi^n}$, from (\ref{hermiteinner}) we have
\begin{equation}\label{fockinner}
\langle Z_1^{\beta_1} \cdots Z_n^{\beta_n}, Z_1^{\alpha_1} \cdots Z_n^{\alpha_n} \rangle = 2^{l-1} \delta_{\beta_1\alpha_1}\cdots\delta_{\beta_n\alpha_n}, ~~~ l  = \beta_1 + \cdots + \beta_n.
\end{equation}

\subsection{Symplectic spin geometry of K\"ahler manifolds}
\subsubsection{Bundles of symplectic spinors}
Let $(M^{2n}, g, J, \om)$ be a compact K\"ahler manifold.  We now define the main objects of study.
\begin{defin}
The Weyl algebra bundle over $M$, denoted $W(TM)$, is the bundle of algebras generated by $TM$ subject to the relation
\[
vw - wv = -\om(v,w).
\]
\end{defin}
The bundle $W(M)$ inherits a connection $\nabla$ from the Levi-Civita connection $\nabla$ on $TM$ by extending its action on $TM \subset W(M)$ to be a derivation of the algebra structure on $W(M)$.

\begin{defin}\label{spinbundef}
We call a complex vector bundle (of necessairly infinite rank) $S\to M$ a bundle of symplectic spinors provided that for each $x\in M$ we have an action
\[
W(M)_x \times S_x \to S_x, ~~ (u,\psi) \mapsto \sigma(u) \psi,
\]
that is isomorphic to the canonical quantization representation (or possibly a dense subrepresentation) of the Weyl algebra associated to the symplectic vector space $T_x M$ and which varies smoothly over $x$ (i.e. if $v$ is a smooth vector field and $\psi$ a smooth section of $S$ then $\sigma(v)\psi$ is also smooth).  We also require $S$ to be equipped with a hermitian metric $\langle \cdot, \cdot \rangle$ and hermitian connection $\nabla^S$,  with the following compatibility relations:
\begin{gather*}
\langle \sigma(v) \psi, \phi \rangle = -\langle \psi, \sigma(v) \phi \rangle, ~~ v \in TM \subset W(M), ~~\psi,\phi \in S \\
\nabla_w^S (\sigma(v) \psi) = \sigma(\nabla_w v) \psi  + \sigma(v) \nabla^S_w \psi.
\end{gather*}
\end{defin}
We will often omit the superscript on $S$ from $\nabla^S$, as it is always clear from context which connection is being used. \\

There are two main constructions of symplectic spinors that we will employ.  The first uses a metaplectic structure, which is a reduction of the structure group of $TM$ from $Sp(n)$ to $Mp(n)$.  As in the case of spin structures, these exists if and only if the second Stiefel-Whitney class vanishes (if and only if $c_1(M)$ is even) and, if they exist, are classified by $H^1(M;\Z/2)$ \cite{kostant}.  If $M$ has a metaplectic structure then we can use the complex structure to further reduce the frame bundle to a principal $\tilde U(n)$ bundle.  

\begin{defin}
Suppose $M$ has a metaplectic structure $P_{\tilde U(n)} \to M$.  Then the bundle of metaplectic spinors is defined to be
\[
S_m = P_{\tilde U(n)} \times_m L^2(\R^n).
\]
The $L^2$ inner product induces a hermitian metric on $S_m$ and the connection on $TM$ determines a unique connection on $P_{\tilde U(n)}$, which further induces a connection on the associated bundle $S_m$.  With these structures, $S_m$ satisfies the conditions of definition \ref{spinbundef} \cite{haber}.
\end{defin} 
\begin{rk}
Since we will only be concerned with smooth sections, we could also define $S_m$ as $P_{\tilde U(n)} \times_m \mathcal S (\R^n)$, where $\mathcal S(\R^n)$ is the Schwartz space.  Indeed, any smooth section of $P_{\tilde U(n)} \times_m L^2(\R^n)$ is actually a section of the subbundle $P_{\tilde U(n)} \times_m \mathcal S(\R^n)$ \cite{haber}.
\end{rk}

The other way to construct symplectic spinors is by using the polarization of $TM \otimes \C$ determined by the complex structure (i.e. using the Fock representation): 
\begin{defprop}
The bundle of Fock spinors is $S_F := S^* T^{1,0} M$.  This is a spinor bundle with action and hermitian metric given in section \ref{focksection}, and connection induced by the Levi-Civita connection (i.e. making the connection a derivation with respect to the symmetric product).
\end{defprop}
\begin{proof}
What needs to be proved is that the connection is compatible with the hermitian metric and the Weyl algebra action.  Let $\{Z_j, \bar Z_j\}$ be a local unitary frame of $TM\otimes \C$.  Since the functions $\langle Z_{i_1} \cdots Z_{i_l}, Z_{j_1} \cdots Z_{j_l}\rangle$ are constant, to check compatibility with the metric one must show that $\langle \nabla (Z_1^{\beta_1} \cdots Z_n^{\beta_n}), Z_1^{\alpha_1} \cdots Z_n^{\alpha_n}\rangle = - \langle Z_1^{\beta_1} \cdots Z_n^{\beta_n}, \nabla (Z_1^{\alpha_1} \cdots Z_n^{\alpha_n})\rangle$.  Letting $\nabla Z_i = \Gamma_i^j Z_j$ for $\Gamma = (\Gamma_i^j) \in \Om^1(M; \mathfrak{u(n)})$, the compatibility condition is a straightforward computation using (\ref{fockinner}) and the fact that $\Gamma_i^j = - \overline{\Gamma_j^i}$. \\

Now for compatibility with the Weyl algebra action, by definition of the action and connection on $S_F$, it is immediate that
\[
\nabla (\sigma(Z_j) \psi) = \sigma(\nabla Z_j) \psi + \sigma(Z_j) \nabla \psi, ~~ \psi \in \Gamma(S^* T^{1,0} M).
\]
To show that $[\nabla, \sigma(\bar Z_j)] = \sigma(\nabla \bar Z_j)$, it is sufficient to check that these agree on local sections of the form $Z_{i_1} \cdots Z_{i_l}$.  For these sections $\nabla = -\nabla^*$ since the functions $\langle Z_{i_1} \cdots Z_{i_l}, Z_{j_1} \cdots Z_{j_l}\rangle$ are constant and $\nabla$ is compatible with the hermitian metric.  As real vectors act anti-self-adjointly, we have $\sigma(V)^* = -\sigma(\bar V)$, $V \in T^{1,0} M$.  Thus
\begin{gather*}
[\nabla, \sigma(\bar Z_j)] = [\nabla^*, \sigma(Z_j)^*] = [\nabla, \sigma(Z_j)]^* = (\sigma(\nabla Z_j))^* \\
= -\sigma(\overline{\nabla Z_j}) = \sigma(\bar Z_j).
\end{gather*}
\end{proof}

The analogy to orthogonal spin geometry is that the Fock spinors on $M$ are the counterparts to the complex spinors on $M$ induced by the complex structure (the construction we gave in section \ref{orthspin}), while the metaplectic spinors are the counterparts to spinors on $M$ induced by a spin structure.  \\

Now let $S$ be any symplectic spinor bundle.  Since $H_0$ lies in the center of $\mathfrak u(n)$, it gives rise to a global operator $H \in \End(S_m)$ given by
\[
H \cdot [p, \psi] = [p, \sigma(H_0)\psi].
\]
We therefore have the decomposition $S = \bigoplus_{l=0}^\infty E_l$ where $E_l$ is a rank $n + l -1 \choose l$ bundle on which $H$ acts as $-(l+n/2)$.  In the case of $S = S_m$, 
\[
E_l = P_{\tilde U(n)} \times_m E_l^0.
\]
Since $\nabla J = 0$ it follows that $\nabla H = 0$ \cite{haber}.  Thus this decomposition of $S$ is preserved by the connection.\\

Given a line bundle $L$ with hermitian metric, we can twist any spinor bundle $S$ to get the spinor bundle $S \otimes L$ (the Weyl algebra acts trivially on the $L$ factor).  By (\ref{fockmeta}), if $M$ is metaplectic then we can get the Fock spinors from the metaplectic spinors in this way by tensoring with a square root of the anti-canonical line bundle (which exists if and only if $w_2(M) = 0$).

\subsubsection{Symplectic Dirac operators}
Let $(S, \nabla^S)$ be any bundle of symplectic spinors.  

\begin{defin}
The symplectic Dirac operators, $D$ and $\tilde D$ are
\begin{gather*}
D: \Gamma(S) \stackrel{\nabla}{\to} \Gamma(S\otimes T^*M) \stackrel{\om}{\to} \Gamma(S \otimes TM) \to \Gamma(S) \\
\tilde D: \Gamma(S) \stackrel{\nabla}{\to} \Gamma(S\otimes T^*M) \stackrel{g}{\to} \Gamma(S \otimes TM) \to \Gamma(S),
\end{gather*}
where the last map, in each case, is the Weyl algebra action.  
\end{defin}
That is, $\tilde D$ is defined in complete analogy to how the Dirac operator is defined in orthogonal spin geometry, whereas in $D$ we use the symplectic form to lower indices instead of the metric.  If $\{a_j, b_j = Ja_j\}$ is a local unitary frame then
\begin{equation}\label{locform}
D = \sum_{j=1}^n (a_j \cdot \nabla^S_{b_j} - b_j \cdot \nabla^S_{a_j}), ~~ \tilde D = \sum_{j=1}^n (a_j \cdot\nabla^S_{a_j} + b_j \cdot \nabla^S_{b_j}).
\end{equation}
For $\xi \in T^* M$, let $\xi^\#$ denote the corresponding vector under the isomorphism $T^*M \to TM$ induced by the symplectic form $\om$.  Then the symbols of the operators $D$ and $\tilde D$ are \cite{haber}
\begin{equation}\label{syms}
\begin{aligned}
\operatorname{sym} D : T^*M \to \End(S), ~~ \xi \mapsto \sigma(\xi^\#), \\
\operatorname{sym} \tilde D: T^*M \to \End(S), ~~ \xi \mapsto \sigma(J\xi^\#),
\end{aligned}
\end{equation}
where, as before, $\sigma$ denotes the Weyl action of a tangent vector. \\

We have the following important propositions, which are proved in \cite{haber}
\begin{prop} \label{selfad}
The operators $D$ and $\tilde D$ are (formally) self-adjoint.
\end{prop}
The self-adjointness of $D$ uses the K\"ahler condition $\nabla J = 0$.  \\

Unlike in the orthogonal case, the symplectic Dirac operators are not elliptic; in the orthogonal case ellipticity follows from the fact that a non-zero tangent vector $v$ is invertible in the Clifford algebra, which is not true in the Weyl algebra.  However, there is a naturally associated second order elliptic operator:
\begin{prop}
The operator $P := i[\tilde D, D]$ is elliptic.
\end{prop}
\begin{proof}
From (\ref{syms}), the symbol of $P$ at $\xi$ is given by 
\[
i[\sigma(J\xi^\#), \sigma(\xi^\#)] = \om(J\xi^\#, \xi^\#) = g(\xi^\#, \xi^\#),
\]
with the first equality coming from (\ref{com}).
\end{proof}

We will now define the symplectic Dolbeault operators.  First we have the following fact proved in \cite{haber}
\begin{prop}
The symplectic Dirac operators satisfy the commutation relations
\[
[H, D] = i\tilde D, ~~ [H, D] = -iD.
\]
\end{prop}
This motivates
\begin{defin}
The symplectic Dolbeault operators are
\[
\D := D + i \tilde D, ~~~ \bar\D = D - i \tilde D.
\]
\end{defin}
This definition should be compared to (\ref{orthdol}). \\

The following properties of these operators are easily proved using the previous propositions
\begin{prop}\label{symdol}
The symplectic Dolbeault operators satisfy the following
\begin{enumerate}
\item $[H, \D] = \D, ~~ [H, \bar\D] = -\bar \D$, i.e. 
\begin{gather*}
\D: \Gamma(E_l) \to \Gamma(E_{l-1}) \\
\bar\D : \Gamma(E_l) \to \Gamma(E_{l+1}).
\end{gather*}
\item $[\D, \bar\D] = 2P$.
\item $\D^* = \bar\D$.
\item $\D$ and $\bar \D$ have symbols
\begin{gather*}
T^*M \to \End(S), ~~ \xi \mapsto \sigma(\xi^\# + i J \xi^\#), \\
T^*M \to \End(S), ~~ \xi \mapsto \sigma(\xi^\# -i J \xi^\#),
\end{gather*}
respectively.
\end{enumerate}
\end{prop}

As $D$ and $\tilde D$ are not elliptic their kernels may not be finite dimensional.  However, we see the following
\begin{cor}
The restrictions of $D$ and $\tilde D$ to $\Gamma(E_l)$ have finite dimensional kernel.
\end{cor}
\begin{proof}
We have $D = \frac{1}{2}(\D + \bar \D)$.  Since $\tilde D$ and $\tilde \bar D$ maps $\Gamma(E_l)$ to the disjoint subspaces $\Gamma(E_{l-1})$ and $\Gamma(E_{l+1})$, respectively, it follows that
\[
\ker D\vert_{E_l} = \ker \D \vert_{E_l} \cap \ker \bar\D \vert_{E_l}.
\]
But
\[
(\ker \D \vert_{E_l} \cap \ker \bar\D \vert_{E_l}) \subset \ker [\D, \bar\D] \vert_{E_l} = \ker P.
\]
Since $P$ is elliptic, $\ker P$, and thus the subspace $\ker D\vert_{E_l}$, is finite dimensional.
\end{proof}
From (\ref{locform}) we see that in a local unitary frame $\{a_j, b_j = Ja_j\}$ the Dolbeault operators take the forms
\[
\D = 4i \sum_{j=1}^n \sigma(\bar{Z_j}) \nabla^S_{Z_j}, ~~ \bar\D = -4i \sum_{j=1}^n \sigma(Z_j) \nabla^S_{\bar Z_j},
\]
where $Z_j = \frac{1}{2}(a_j - i b_j) \in T^{1,0} M$ and $\bar Z_j = \frac{1}{2}(a_j + i b_j) \in T^{0,1} M$.  \\

From these formulas we expect holomorphicity to play a role.  While the Fock spinors, being $S^* T^{1,0} M$, have a natural holomorphic structure for which the spinor connection coincides with the Chern connection (i.e. the unique connection compatible with the hermitian structure and with $(0,1)$ part equal to $\bar\partial$, the holomorphic structure), the holomorphicity of the metaplectic spinors is not so immediate.  This is taken care of by the following
\begin{prop}\label{chernconn}
The bundle $S_m$ has a natural holomorphic structure for which the spinor connection is the Chern connection.
\end{prop}
\begin{proof}
We have, by definition,
\[
E_l = P_{\tilde U(n)} \times_{m_l} E_l^0,
\]
where $m_l$ denotes the restriction to $E_l^0$ of the metaplectic representation.  Now by Weyl's unitarian trick, the unitary representation $m_l$ extends to a unique holomorphic representation $\tilde m_l$ of $\widetilde{GL}(n,\C)$ on $E_l^0$, where $\widetilde{GL}(n,\C)$ is the double cover of $GL(n,\C)$.  Thus $E_l = P_{\widetilde{GL}(n,\C)} \times_{\tilde m_l} E_l^0$ is holomorphic (having holomorphic transition functions $\tilde m_l \circ h_{\alpha\beta}$, where $h_{\alpha\beta}$ are holomorphic transition functions for $TM$). \\

The spinor connection is automatically compatible with the hermitian structure on $S_m$ since it is associated to $P_{\tilde U(n)}$ via a unitary representation.  Thus we are just left to show that its $(0,1)$ part is $\bar\partial_{S_m}$.  In a local trivialization $U \subset M$, the connection on $M$ is of the form $\nabla = d + A$ where $A \in \Om^{1,0}(U;\mathfrak{u}(n))$.  Then over $U$, $\nabla^{S_m} = d + \pi(A)$, which has $(0,1)$ part $\bar\partial_{S_m}$.
\end{proof}

\begin{cor}\label{holvac}
For $S = S_m$ or $S_F$, we have
\[
(\text{holomorphic sections of $E_l$}) \subseteq \ker \bar\D\vert_{E_l},
\]
which equality for $l=0$.
\end{cor}
\begin{proof}
The inclusion is immediate from the proposition and the local form of $\bar\D$.  Equality for $l=0$ follows from the local form of $\bar\D$ together with the algebraic fact that for $Z_1, \ldots, Z_n\ne 0$, the maps $\sigma(Z_j) : E_0^0 \to E_1^0$ take $1 \mapsto h_{0\cdots 0 1 0 \cdots 0}$, and therefore have independent images.
\end{proof}

\section{Riemann surfaces}
The case of complex dimension 1 is unique because here $\operatorname{rk} E_l = 1$ for all $l$ and we have following
\begin{prop}
For $M^2$, the operators $\D\vert_{E_l}$ and $\bar\D\vert_{E_l}$ are elliptic (assuming $l\ne 0$ in the case of $\D$).
\end{prop}
\begin{proof}
We will show that if $(V, J_0, \om_0, g_0)$ is a two-dimensional hermitian vector space, then 
\[
\sigma(v+iJ_0v)\sigma(v-iJ_0v) : E_l^0 \to E_l^0
\]
is a non-zero scalar, from which the proposition will follow by proposition \ref{symdol}(4).  We have
\begin{align*}
\sigma((v+iJ_0v)(v- i J_0v)) &= \sigma(v^2 + (J_0v)^2 + i[J_0v,v]) \\
&= \sigma(v^2 + (J_0v)^2 + i \om(v,J_0v)) \\
&= \sigma(v^2 + (J_0v)^2 + ig(v,v)) \\
&= \sigma(v^2 + (J_0v)^2) - g(v,v).
\end{align*}
A straightforward calculation using a symplectic basis $\{a,b\}$ such that $\sigma(a) = it, \sigma(b) = \frac{d}{dt}$ shows that $\sigma(v^2 + (J_0v)^2) = 2g_0(v,v) H$.  Thus the above is equal to
\begin{gather*}
2g_0(v,v) H_0 - g_0(v,v) = -2g_0(v,v)(l+1/2) - g_0(v,v) \\
= -(2l+2) g_0(v,v) \ne 0.
\end{gather*}
\end{proof}

\begin{cor}(of proof)\label{kerDRS}
We have $\ker \bar\D\vert_{E_l}$ is exactly the space of holomorphic sections of $E_l$.
\end{cor}
\begin{proof}
Locally, we have $\bar \D = -4i\sigma(Z) \nabla_{\bar Z}$  where $Z$ is a basis for $T^{1,0} M$ and $\bar Z$ a basis for $T^{0,1} M$.  By the proof of the previous proposition, $\sigma(Z)$ is invertible so $\ker \bar\D = \{\psi \in \Gamma(E_l) : \nabla_{\bar Z} \psi = 0\}$, which is the space of all holomorphic sections since $\nabla$ is the Chern connection (Lemma \ref{chernconn}).
\end{proof}

Similarly,
\begin{cor}(of proof)
The space $\ker D\vert_{E_l} = \ker \bar D \vert_{E_l}$ is the set of all parallel sections of $E_l$.
\end{cor}

We will now compute the indices of the operators $\D$ and $\bar\D$ for $M_g$, a Riemann surface of genus $g$.  We have 
\[
\langle c_2(M_g), [M_g]\rangle = \langle e(M_g), [M_g]\rangle = \chi(M_g) = 2-2g
\]
is even, so that $M_g$ admits metaplectic structures.  Fix a metaplectic structure $P_{\tilde U(1)} \to M_g$ (of course $\tilde U(1) \simeq U(1)$).  

\begin{prop}\label{index}
For the metaplectic spinors, $S_m$, we have
\[
index~ \bar\D\vert_{E_l} = \dim\ker \bar\D\vert_{E_l} - \dim\ker \D\vert_{E_{l+1}} = (2l+2) (1-g)
\]
and for the Fock spinors, $S_F$, we have
\[
index~ \bar\D\vert_{E_l} = \dim\ker \bar\D\vert_{E_l} - \dim\ker \D\vert_{E_{l+1}} = (2l+1) (1-g).
\]
\end{prop}

Before proving this, we first present a quick consequence.  Since the space of holomorphic functions on a Riemann surface is one-dimensional, we have (for the Fock spinors) that $\dim\ker \bar\D\vert_{E_0} = 1$.  Since $E_1 \simeq TM_g$, the above theorem says that the space of all holomorphic vector fields $X$ with $\nabla_Z X = 0$ for $Z \in T^{1,0} M_g$ has dimension $g$.  Since $\nabla$ preserves the metric, this means that $\nabla_Z X^\flat = 0$, where $X^\flat \in \Om^{0,1}(M_g)$ is the image of $X$ under the isomorphism $\Gamma(TM_g) \simeq \Gamma(T^{1,0} M_g) \simeq \Om^{0,1}(M_g)$ induced by the metric.  Conjugating, we have proven the classical result
\begin{cor}
The dimension of the space of holomorphic sections of the canonical line bundle $T^*M_g \simeq \Lambda^{1,0} M_g$ is $g$.
\end{cor}
\begin{proof}[Proof of proposition]
Let $\pi : T^* M_g \to M_g$.  We consider first the case of the metaplectic spinors.  The symbol class of the operator $\D \vert_{E_l} : \Gamma(E_l) \to \Gamma(E_{l+1})$ is the sequence
\[
\pi^* E_l \to \pi^* E_{l+1}, ~~ (\xi, \psi) \mapsto \sigma(\xi^\# - i J\xi^\#)\psi
\] 
of vector bundles over $T^*M_g$.  This lies in $K(T^*M_g)$, the $K$ theory of $T^*M_g$.  Now we have
\[
E_l = P_{\tilde U(1)} \times_m E_l^0, ~~ TM_g = P_{\tilde U(1)} \times_p \C,
\]
where $p$ denotes the double cover map $\tilde U(1) \to U(1)$.  The $\tilde U(1)$ equivariant map $\sigma : \C \to \Hom(E^0_l, E^0_{l+1}), v \mapsto \sigma(v - i Jv)$ determines an element in $K_{\tilde{U}(1)}(\C)$ and the symbol class of $\bar\D\vert_{E_l}$ is the image of this element under the map $K_{\tilde U(1)}(\C) \to K(TM_g)$.  In this case the symbol class is said to be associated to a $\tilde U(1)$-structure and there is a simple version of the Atiyah-Singer index theorem (the index theorem for $H$-structures) that computes its index \cite{atiyahsinger}.  In our case, since the Todd class of $M_g$ is just 1, the formula is
\[
\text{index} \bar\D\vert_{E_l} = \left\langle\frac{\ch E_{l+1}^0 - \ch E_l^0}{p^*(\alpha)}(P_{\tilde U(1)}), [M_g]\right\rangle.
\]
where $\alpha \in \mathfrak u(1)^*$ is the generator of the weight lattice of $\mathfrak u(1)$, which represents the first Chern class functor $c_1$.  Let $x$ be the generator of $H^2(M_g;\Z) \simeq \Z$ dual to the fundamental class $[M_g] \in H_2(M_g;\Z)$.  A subtle issue is that as a representation of $U(1)$, the weight of $E_l^0$ is $-(2l+1) \alpha$ since we need to identify $\operatorname{Lie}(\tilde U(1))$ with $\operatorname{Lie}(U(1))$ via multiplication by 2.  From this and the fact that $c_1(P_{\tilde U(1)}) = (1 - g)x$ since $c_1(M_g) = (2-2g)x$, we have
\begin{gather*}
\text{index} \bar\D\vert_{E_l} = \left\langle \frac{\frac{1}{2}(2l+3)^2 c_1^2 - \frac{1}{2} (2l+1)^2 c_1^2)}{2c_1}(P_{\tilde U(1)}, [M_g] \right\rangle \\
= \frac{1}{4} (8l + 8) \langle c_1(P_{\tilde U(1)}), [M_g] \rangle = (2l+2)(1-g).
\end{gather*}

We omit the computation for the Fock spinors since it is similar but slightly easier because now everything is associated to the principal $U(1)$ frame bundle and $p$ is just the identity map.  
\end{proof}

\section{Flag manifolds}
\subsection{Set-up}
Let $G$ be a simply-connected compact semi-simple Lie group of rank $k$ and fix a maximal torus $T$ and positive roots $\{\alpha_1, \ldots, \alpha_n\}$ such that the $\{\alpha_1,\ldots,\alpha_k\}$ are the simple roots.  Let $g_0$ be negative the Killing form.  Then we can find an orthonormal basis $\{E_{\alpha_j}, F_{\alpha_j}\}$ of $\t^\perp$ such that $Z_{\alpha_j} := \frac{1}{2}(E_{\alpha_j} -iF_{\alpha_j})$ and $\bar Z_{\alpha_j} := \frac{1}{2}(E_{\alpha_j} + iF_{\alpha_j})$ are root vectors for $\alpha_j$ and $-\alpha_j$, respectively.  We define a complex structure on $\t^\perp$ by
\[
J_0 : \t^\perp \to \t^\perp, ~~ E_\alpha \mapsto F_\alpha, ~~ F_\alpha \mapsto -E_\alpha.
\]
We define $\om_0$ a symplectic form on $\t^\perp$ by $\om_0(X,Y) = g_0(JX, Y)$.  Now we have
\[
T(G/T) = G \times_{Ad} \t^\perp
\]
and $g_0, J_0$ and $\om_0$ all become global objects, $g, J, \om$, making $(G/T, g, J, \om)$ into 
a Hermitian manifold.  This Hermitian structure is actually not K\"ahler but is a K\"ahler with torsion structure.  Thus instead of the Levi-Civita connection we use the canonical connection on $G/T$, under which $J, g$ and $\om$ are parallel.  Upon lowering indices, the torsion of this connection is
\[
(X,Y,Z) \mapsto -g_0([X,Y],Z), ~~ X, Y, Z \in \mathfrak t^\perp \simeq T_{eT} G/T ~~\text{    \cite{kobnom}}.
\]
This connection is also characterized as the unique connection with skew-symmetric torsion that preserves $(g, J, \om)$ \cite{friedivan,gaud}.  \\

Since we are not in the K\"ahler case, we must be careful with the results of section 2, which assume a K\"ahler structure.  However, with the exception of proposition \ref{selfad}, the proofs of all of the statements only assume that the connection used preserves the Hermitian structures.  For connections with torsion and parallel complex structure, a sufficient condition for $D$ and $\tilde D$ to be self-adjoint is the vanishing of the torsion vector field, defined by
\[
\mathcal T = \sum_{j=1}^n \mathbf T(a_j,b_j),
\]
where $\mathbf T$ is the torsion tensor and $\{a_1,\ldots,a_n, b_1,\ldots,b_n\}$ is a symplectic frame \cite{haber}.  In the case of flag manifolds, a symplectic basis at $eT$ is proportional to $\{Z_\alpha, Z_{-\alpha}\}$ where $Z_\alpha$ is a root vector for $\alpha$.  Since $\mathbf T(Z_\alpha, Z_{-\alpha}) = -[Z_\alpha, Z_{-\alpha}]_{\t^\perp} = 0$, we see that $\mathcal T$ vanishes.  Thus proposition \ref{selfad} continues to hold. \\

We note for future reference that
\begin{equation}\label{algeqs}
g_0(Z_\alpha, \bar Z_\beta) = \frac{1}{2}\delta_{\alpha \beta}, ~~ \om_0(Z_\alpha, \bar Z_\beta) = \frac{i}{2}\delta_{\alpha\beta}, ~~ [Z_{\alpha_j}, \bar Z_{\alpha_j}] = \frac{1}{2} H_{\alpha_j},
\end{equation}
 where $H_{\alpha_j} \in \t\otimes\C$ is dual to $\alpha_j$. \\

Since $G$ is simply connected, the map $G \to U(\g)$ lifts to $G \to \tilde U(\g)$.  Thus the map $T \to U(\t^\perp)$ lifts to $\tilde{Ad}: T \to \tilde U(\t^\perp)$ and $G/T$ has a metaplectic structure with principal $\tilde U(n)$ frame bundle
\[
G \times_{\tilde{Ad}} U(\t^\perp).
\]
This metaplectic structure is unique since $H^1(G/T;\Z/2)$ is trivial since $G/T$ is simply-connected. \\

The set $\{iH_\alpha : \alpha \text{ simple }\}$ is a basis for $\t$ and
\[
\text{ad}_{iH_\alpha}(E_\beta) = \beta(H_\alpha) F_\beta = g_0(\alpha,\beta) F_\beta.
\]
Thus
\begin{equation}\label{admap}
\text{ad} : \t \to \tilde{\mathfrak u}(\t^\perp) \subset W(\t^\perp), ~~ iH_\alpha\mapsto \frac{1}{2}\sum_{j=1}^n g_0(\alpha,\alpha_j) (E_{\alpha_j}^2 + F_{\alpha_j}^2).
\end{equation}
It follows from this and (\ref{weights}) that the weights of the representation
\[
T \stackrel{\tilde{Ad}}{\to} U(\t^\perp) \stackrel{m}{\to} U(E_l^0)
\]
are
\begin{equation}\label{weights}
\mu_{\beta_1 \ldots \beta_n}  = \sum_{i=1}^k \sum_{j=1}^n 2\frac{g(\alpha_i, \alpha_j)}{g(\alpha_i, \alpha_i)} (\beta_j + \frac{1}{2}) \om_i
\end{equation}
where $\beta_1 + \cdots + \beta_n = l$ and $\{\om_1,\ldots,\om_k\}$ are the fundamental weights, i.e.
\[
\om_i\left(\frac{2}{g(\alpha_j,\alpha_j)}H_{\alpha_j}\right) = \delta_{ij}, ~~ 1 \le i,j \le k.
\]
Later we will need
\begin{lem}\label{rho}
The weight $\mu_{0\cdots0}$ is equal to $\rho := \sum_{i=1}^k \om_i$.
\end{lem}
\begin{proof}
We have
\begin{gather*}
\mu_{0\cdots0} = \sum_{i=1}^k \frac{1}{g(\alpha_i,\alpha_i)}g\left(\alpha_i,\sum_{j=1}^n \alpha_j\right) \om_i.
\end{gather*}
But, as is well known, $\frac{1}{2}\sum_{j=1}^n \alpha_j = \rho$.  Thus this is equal to
\begin{gather*}
\sum_{i=1}^k \frac{2}{g(\alpha_i,\alpha_i)}g\left(\alpha_i,\sum_{j=1}^k \om_j\right) = \sum_{i,j=1}^k g\left(\frac{2}{g(\alpha_i,\alpha_i)} \alpha_i, \om_j\right) \om_i \\
= \sum_{i,j=1}^k \delta_{ij} \om_i = \rho.
\end{gather*}
\end{proof}

The spinor bundles are given by
\[
E_l = G \times_{m \circ \tilde{Ad}} E_l^0.
\]
We will use the following identification
\[
\Gamma(E_l) \simeq \{ f : G \to E_l^0 ~ | f(gt) = m\circ \tilde{Ad}(t^{-1}) f(g), ~~ t \in T\}.
\]
Under this identification and the isomorphism $T(G/T) \simeq G\times_{Ad} \t^\perp$, the connection on these bundles is given by
\[
(\nabla_{[g,X]} f)(g) = X_g f,
\]
where on the right hand side we view $X \in \t^\perp$ as a left-invariant vector field on $G$, differentiating the function $f$ at the point $g \in G$. \\

The symplectic Dirac and Dolbeault operators are all $G$-equivariant and are given by
\begin{equation}\label{locformflag}
\begin{aligned}
D\phi &= \sum_{\alpha > 0} \left(\sigma(E_\alpha) F_\alpha \cdot \phi - \sigma(F_\alpha)E_\alpha\cdot \phi\right) \\
\tilde D \phi &= \sum_{\alpha > 0} \left(\sigma(E_\alpha) E_{\alpha} \cdot \phi+ \sigma(F_\alpha) F_\alpha \cdot \phi\right) \\
\D \phi &= i \sum_{\alpha > 0} \sigma(E_\alpha + i F_\alpha) (E_\alpha - i F_\alpha) \cdot \phi= 4i\sum_\alpha \sigma(\bar Z_\alpha) Z_\alpha \\
\bar\D \phi &= -i \sum_{\alpha > 0} \sigma(E_\alpha - i F_\alpha) (E_\alpha + i F_\alpha)\cdot \phi = -4i\sum_\alpha \sigma(Z_\alpha) \bar Z_\alpha,
\end{aligned}
\end{equation}
where $X \cdot\phi$ denotes differentiation of $\phi$ with respect to the left-invariant vector field $X$. \\

Being $G$-equivariant means that the kernels, cokernels, and indices of these operators lie in $R(G)$, the representation ring of $G$. 

\subsubsection{Twisting by an equivariant line bundle}
All equivariant line bundles on $G/T$ are given by $L_\mu = G\times_\mu T$, where $\mu$ is a weight.  Since
\[
L_{\mu_1} \otimes L_{\mu_2} \simeq L_{\mu_1 + \mu_2},
\]
if we twist our spinors by tensoring each $E_l$ with $L_{\mu - \rho}$, we will get vacuum state $(l=0)$ equal to $L_\mu$.  We denote the corresponding Dirac and Dolbeault operators $\D_\mu, \bar\D_\mu$. \\

We can now readily compute $\ker \bar\D_\mu\vert_{E_0} \in R(G)$.
\begin{prop}\label{kerD}
We have
\[
\ker\bar\D_\mu\vert_{E_0} = V_\mu,
\]
where $V_\mu$ denotes the representation of $G$ with highest weight $\rho$.  In particular, for the untwisted Dirac operator we have
\[
\ker \bar\D\vert_{E_0} = V_\rho
\] 
and the dimension of this space is $2^{\dim_\C G/T}$.
\end{prop}
\begin{proof}
By corollary \ref{holvac}, $\ker \bar\D_\mu\vert_{E_0}$ is the space of holomorphic sections of $E_0$.  Since $E_0 = G \times_\mu \C$, it follows from the Borel-Weil theorem that this space, as a $G$ representation, is $V_\mu$. \\

The dimension for $V_\rho$ comes from the Weyl dimension formula:
\[
\dim V_\rho = \prod_{\alpha > 0} \frac{g_0(\rho + \rho, \alpha)}{g_0(\rho,\alpha)} = 2^{\#\{\alpha > 0 \}} = 2^{\dim_\C G/T}.
\]
\end{proof}

Since 
\[
P \vert_{E_0} = \frac{1}{2}(\D\bar\D\vert_{E_0} - \bar\D \D \vert_{E_0}) = \frac{1}{2}\D\vert_{E_1} \bar\D\vert_{E_0} = \frac{1}{2}\bar\D\vert_{E_0}^* \bar\D\vert_{E_0}
\]
we immediately get
\begin{cor}
We have
\[\ker P_\mu \vert_{E_0} = \ker \bar \D_\mu \vert_{E_0} = V_\mu.\]
\end{cor}

\subsection{A formula for $P$}
We focus now on the elliptic operator $P_\mu = \frac{1}{2}[\D_\mu,\bar\D_\mu]$.
\begin{prop}\label{Pform}
Let $\Om_\g$ be the Casimir element of $\g$ (viewed as as second order differential operator on $E_l$) and $\Om_\t = \sum_{j=1}^k h_j^2$ the Casimir element of $\t$ (where $\{h_j\}$ is an orthonormal basis for $\t$).  Then
\begin{align*}
P_\mu = -\Om_\g + \Om_{\t} + 2\sum_{\alpha > 0} (\sigma(Z_\alpha) \sigma(\bar Z_\alpha) +& \sigma(\bar Z_\alpha)\sigma(Z_\alpha)) H_\alpha  \\
&- 8 \sum_{\alpha\ne\beta > 0} \sigma(Z_\alpha) \sigma(\bar Z_\beta) [\bar Z_\alpha, Z_\beta]
\end{align*}
and
\[
P_\mu\vert_{E_0} = -\Om_\g + ||\mu||^2 + 2g_0(\rho, \mu) = -\Om_\g + ||\mu + \rho||^2 - ||\rho||^2,
\]
where $|| \cdot ||^2 = g(\cdot,\cdot)$.  \\
\end{prop}
\begin{rk} Our Casimir's are negative of the usual ones associated to the Killing form, since our metric is minus the Killing form.  One observes that the term $||\mu + \rho||^2 - ||\rho||^2$ is the infinitesimal character of the Casimir $\Om_\g$ for the representation $V_\mu$.
\end{rk}
\begin{proof}
By (\ref{locformflag}) $P_\mu = 8\sum_{\alpha,\beta > 0} [\sigma(\bar Z_\alpha) Z_\alpha, \sigma(Z_\beta)\bar Z_\beta]$.  We have
\begin{align*}
2 [\sigma(Z_\alpha)& \bar Z_\alpha, \sigma(\bar Z_\beta) Z_\beta] \\
&= \sigma(Z_\alpha)\sigma(\bar Z_\beta) \bar Z_\alpha Z_\beta + \sigma(Z_\alpha)\sigma(\bar Z_\beta) \bar Z_\alpha Z_\beta - \sigma(\bar Z_\beta)\sigma(Z_\alpha) Z_\beta \bar Z_\alpha \\
 &\hspace{.65in} - \sigma(\bar Z_\beta)\sigma(Z_\alpha) Z_\beta \bar Z_\alpha \\
&= \sigma(Z_\alpha)\sigma(\bar Z_\beta) (Z_\beta \bar Z_\alpha + [\bar Z_\alpha, Z_\beta]) + \sigma(Z_\alpha)\sigma(\bar Z_\beta) \bar Z_\alpha Z_\beta \\ 
 &\hspace{.65in} - \sigma(\bar Z_\beta)\sigma(Z_\alpha) (\bar Z_\alpha Z_\beta + [Z_\beta, \bar Z_\alpha]) - \sigma(\bar Z_\beta)\sigma(Z_\alpha) Z_\beta \bar Z_\alpha \\
&= [\sigma(Z_\alpha), \sigma(\bar Z_\beta)] (Z_\beta \bar Z_\alpha + \bar Z_\alpha Z_\beta) +\\
&\hspace{.65in} (\sigma(Z_\alpha) \sigma(\bar Z_\beta) + \sigma(\bar Z_\beta)\sigma(Z_\alpha)) [\bar Z_\alpha, Z_\beta] \\
&= -i\om(Z_\alpha, \bar Z_\beta)(Z_\beta \bar Z_\alpha + \bar Z_\alpha Z_\beta) + (\sigma(Z_\alpha) \sigma(\bar Z_\beta) \\
&\hspace{.65in}+ \sigma(\bar Z_\beta)\sigma(Z_\alpha)) [\bar Z_\alpha, Z_\beta] \\
&= \frac{1}{2} \delta_{\alpha\beta} (Z_\beta \bar Z_\alpha + \bar Z_\alpha Z_\beta) + (\sigma(Z_\alpha) \sigma(\bar Z_\beta) + \sigma(\bar Z_\beta)\sigma(Z_\alpha)) [\bar Z_\alpha, Z_\beta].
\end{align*} 
Thus multiplying the above by -4 and using (\ref{algeqs}),
\begin{align*}
P_\mu &= -2\sum_{\alpha > 0} (Z_\alpha \bar Z_\alpha +  \bar Z_\alpha Z_\alpha) - 4\hspace{-4pt}\sum_{\alpha,\beta > 0} \hspace{-4pt}(\sigma(Z_\alpha) \sigma(\bar Z_\beta) + \sigma(\bar Z_\beta)\sigma(Z_\alpha)) [\bar Z_\alpha, Z_\beta] \\
&= -\Om_\g + \Om_{\t} + 2 \sum_{\alpha > 0} (\sigma(Z_\alpha) \sigma(\bar Z_\alpha) + \sigma(\bar Z_\alpha)\sigma(Z_\alpha)) H_\alpha \\
&\hspace{1in}-8\sum_{\alpha\ne\beta > 0} \sigma(Z_\alpha) \sigma(\bar Z_\beta) [\bar Z_\alpha, Z_\beta].
\end{align*}
The last term follows from the fact that $\sigma(\bar Z_\beta)$ and $\sigma(Z_\alpha)$ commute since $\om(\bar Z_\beta, Z_\alpha) = 0$ for $\alpha \ne \beta$.  \\

For $P_\mu\vert_{E_0}$ we see from (\ref{ladderops}) that the operators $\sigma(\bar Z_\alpha)$ are identically zero on $E_0$ and each $\sigma(Z_\alpha)\sigma(\bar Z_\alpha) + \sigma(\bar Z_\alpha)\sigma(Z_\alpha)$ acts as $-1/2$.  Thus
\[
P\vert_{E_0} = -\Om_\g + \Om_\t - 2\sum_{\alpha > 0} \frac{1}{2} H_\alpha = -\Om_\g + \Om_\t - 2H_{\frac{1}{2}\sum_{\alpha > 0} \alpha} = -\Om_\g + \Om_\t - 2H_\rho.
\]
By the equivariance condition on sections of $E_0$, $\Om_\t = \sum_j h_j^2$ acts as $\sum_j (-\mu(h_j))^2 = ||\mu||^2$ and $H_\rho$ acts as $-g_0(\rho, \mu)$.
\end{proof}

\begin{prop}\label{Pspec}
As before, let $V_\gamma$ denote the representation of $G$ with highest weight $\gamma$.  The spectrum of $P_\mu\vert_{E_0}$ is 
\[
\operatorname{spec} P_\mu\vert_{E_0} = \{-||\gamma + \rho||^2  + ||\mu + \rho||^2: V_\gamma \text{ contains $\mu$ as a weight}\}.
\]
and the eigenspace with value $\lambda$ is
\[
\bigoplus_{-||\gamma+\rho||^2 + ||\mu + \rho||^2 = \lambda} (\dim V_\gamma(\mu))V_\gamma \in R(G),
\]
where $V_\gamma(\mu)$ denotes the $\mu$ weight space of the representation $V_\gamma$.
\end{prop}
\begin{rk}
Observe that the eigenvalue corresponding to the weight $\gamma$ is the difference of the infinitesimal characters of $\Om_\g$ for the highest weight modules $V_\mu$ and $V_\gamma$.
\end{rk}
\begin{proof}
By proposition \ref{Pform}, we just need to determine the spectrum of $\Om_\g$ acting on $\Gamma(E_0)$.  Recall that a section of $E_0$ is a function $f: G \to \C$ such that 
\begin{equation}\label{eqcond}
f(ge^H) = e^{-\mu(H)} f(g).
\end{equation}
Let $\gamma^*$ denote the highest weight of the representation $V_\gamma^*$.  Then the Peter-Weyl theorem gives us a decomposition
\[
L^2(G) \simeq \bigoplus_{\gamma \text{ dominant}} V_\gamma \otimes V_{\gamma^*}
\]
where the left regular and right regular representation of $G$ on $L^2(G)$ intertwines with the actions on $V_\gamma$ and $V_{\gamma^*}$, respectively.  \\

Since we are acting by left-invariant vector fields, whose flows generate the right regular representation, $\Om_\g$ acts only on the right and the above decomposition is diagonal with respect to $\Om_\g$. As is well-known, $\Om_\g$ acts as $||\gamma^* + \rho||^2 - ||\rho||^2$ on $V_{\gamma^*}$.  Thus each $f_\gamma \in V_\gamma \otimes V_{\gamma^*}$ is an eigenvector for $\Om_\g$ with eigenvalue $||\gamma^* + \rho||^2 - ||\rho||^2$.  But $\gamma^* = -w_0 \gamma$ where $w_0$ is the unique element of the Weyl group mapping the positive Weyl chamber to its negative.  It is well-known that $w_0 \rho = -\rho$ so that
\[
||\gamma^* + \rho||^2 = ||-w_0 \gamma - w_0\rho||^2 = ||\gamma + \rho||^2,
\]
since the inner product is invariant. \\

Now, in order for $f_\gamma$ to be non-zero and satisfy the equivariance condition (\ref{eqcond}), the $V_{\gamma^*}$ factor of $f_\gamma$ must be in the $-\mu$ weight space.  But $-\mu$ is a weight of $V_{\gamma}^*$ if and only if $\mu$ is a weight of $V_\gamma$.
\end{proof}

As a quick application, we show that the spectrum of $P\vert_{E_0}$ retains enough information about the representation theory of $G$ to distinguish between the flag manifolds associated to the Lie algebras $B_n = \mathfrak o(2n+1)$ and $C_n = \sp(n)$.
\begin{cor}
For $n \ge 3$, the flag manifolds $Spin(2n+1)/T$ and $Sp(n)/T$ are not isomorphic as Hermitian manifolds (with metric induced by the Killing form and complex structure from a choice of positive roots).
\end{cor}
\begin{proof}
As mentioned above, the Hermitian structure uniquely determines a connection with skew-torsion that preserves the Hermitian structure (and in this case is the canonical metric).  Thus the symplectic Dolbeault operators are canonically defined.  For the Fock spinors, $P\vert_{E_0} = P_0\vert_{E_0}$ and so by the proposition, the spectrum of $P\vert_{E_0}$ is $\{-||\gamma + \rho||^2 + ||\rho||^2 : V_\gamma \text{ contains 0 as a weight}\}$.  Recalling that we are working with negative the Killing form, the smallest eigenvalue is 0 which occurs for $\gamma = 0$.  Now, for $\mathfrak{o}(2n+1)$, the next smallest value for $-||\gamma + \rho||^2 + ||\rho||^2$ is at $\gamma = \om_1$, the first fundamental weight.  The representation $V_{\om_1}$ is the defining representation of $\mathfrak o(2n+1)$ on $\C^{2n+1}$, has minimal dimension among all non-trivial representations of $\mathfrak o(2n+1)$, and contains 0 as a weight (with multiplicity one).  Thus the eigenspace of the smallest positive eigenvalue for $P\vert_{E_0}$ has dimension $2n+1$.  But for $n \ge 3$, the only representations of $\sp(n)$ of dimension $2n+1$ are $\mathbbm 1 \oplus \C^{2n}$ and $\mathbbm 1 \oplus \cdots \oplus \mathbbm 1$, where $\mathbbm 1$ denotes the trivial representation and $\C^{2n}$ is the defining representation.  By the proposition, these cannot make up an eigenspace since $\C^{2n}$ does not contain $0$ as a weight and the multiplicity of $\mathbbm 1$ in the representation $\mathbbm 1 \oplus \cdots \oplus \mathbbm 1$ is more than the dimension of the 0 weight space of $\mathbbm 1$.
\end{proof}

\section{$\C P^1$}
We now specialize to $\C P^1 = SU(2)/U(1)$.  In this case the Hermitian structure introduced in the previous section is actually K\"ahler and agrees with the usual K\"ahler structure on $\C P^1$.  First we briefly recall the representation theory of $SU(2)$.
\subsection{Representation theory of $SU(2)$}
The complex representations of $SU(2)$ are indexed by the non-negative integers.  For any integer $k \ge 0$, we denote by $V_k$ the irrep of $SU(2)$ with dimension $k+1$.  Let
\[
h = \left(\begin{array}{cc}
1 & 0 \\
0 & -1 
\end{array}\right), ~~ X = \left(\begin{array}{cc}
0 & 1 \\
0 & 0 
\end{array}\right), ~~ Y = \left(\begin{array}{cc}
0 & 0 \\
1 & 0 
\end{array}\right)
\]
Then $h$ is the coroot corresponding to the positive root of $SU(2)$ and $V_k$ splits into a direct sum of one-dimensional eigenspaces for $H$ with the eigenvalues $k, k - 2, \ldots, 2-k, -k$.  Since $-Wt(V_k) = Wt(V_k)$, we have $V_k \simeq V_k^*$.  The following is a standard result.

\begin{prop}\label{su2cas}
The Casimir element (using negative the Killing form) is $\Om := -\frac{1}{8} h^2 - \frac{1}{4} (XY + YX) \in \mathcal U(\mathfrak{sl}(2,\C))$.  The action of $\Om$ on $V_k$ is by the scalar $-\frac{(k+1)^2 - 1}{8}$.
\end{prop}

\subsection{Symplectic spin geometry of $\C P^1$}
By (\ref{weights}), for the metaplectic spinors, $E_l$ is associated to $G\to G/T$ via the weight $(2l+1)\om$.  For the Fock spinors, $E_l = G \times_{S^l Ad} S^l(\t^\perp)$ and so is associated to $G\to G/T$ via the weight $2l\om$.


\begin{prop}\label{kersu2}
For $\C P^1$, we have for the metaplectic spinors
\[
\ker\bar \D \vert_{E_l} = V_{2l+1}, ~~ \ker \D\vert_{E_l} = 0
\]
and for the Fock spinors
\[
\ker\bar \D \vert_{E_l} = V_{2l}, ~~ \ker \D\vert_{E_l} = 0,
\]
where $l \ne 0$ for the equations involving $\D$.
\end{prop} 

\begin{rk}
This is consistent with proposition \ref{index} since $\dim V_k = k+1$.
\end{rk}

\begin{proof}
By (\ref{weights}), for the metaplectic spinors, $E_l$ is associated to $G\to G/T$ via the weight $\mu_l = (2l+1)\om$, where $\om$ is the fundamental weight.  Thus from  (\ref{kerDRS}) and the Borel-Weil theorem, we get $\ker\bar\D\vert_{E_l} = V_{2l+1}$.  In the case of the Fock spinors, $E_l = G \times_{S^l Ad} S^l(\t^\perp)$ and so is associated to $G$ via the weight $2l$. \\

That $\ker\D\vert_{E_l}$ is trivial for $l \ne 0$ follows from Peter-Weyl theorem:  write $f \in \ker \D\vert_{E_l}$ as $f = \sum_\gamma f_\gamma$ with $f_\gamma \in V_\gamma \otimes V_\gamma^*$.  The equivariance condition $f(ge^H) = e^{-\mu_l(H)} f(g)$ implies that the $V_\gamma^*$ factor of $f_\gamma$ lies in the $-\mu_l$ weight space.  The condition that $\D f = 0$ becomes $Z_\alpha f_\gamma = 0$ for all positive roots $\alpha$.  Since $Z_\alpha$ is acting on the $V_\gamma^*$ factor, which lies in the weight space of $-\mu_l$, this can happen if and only if $-\mu_l$ is a highest weight.  This is impossible since $-\mu_l$ is in the (strictly) negative Weyl chamber for $l \ne 0$.
\end{proof}

From now on we will work only with the metaplectic spinors, the case of the Fock spinors being analogous.  We wish to find the spectrum of $P \vert_{E_l}$.  We start off with the following
\begin{prop}\label{Psu2}
We have
\[
P = - \Om - \frac{3}{2} H^2
\]
and
\[
[P, \D] = -3\D H -\frac{3}{2} \D, ~~ [P, \bar \D] = 3\bar\D H - \frac{3}{2}\bar \D.
\]
\end{prop}  

\begin{rk}
The first formula appears in \cite{haber} but with $3/2$ replaced by $12$.  This is because the metric we are using differs from theirs by a factor of 8 and scaling the metric scales $\Om$ and $P$ the same but leaves $H$ unchanged.
\end{rk}

\begin{proof}
Since $SU(2)$ has rank one, we have $H = \sigma(Z_\alpha)\sigma(\bar Z_\alpha) + \sigma(\bar Z_\alpha)\sigma(Z_\alpha)$, where $\alpha$ is the positive root.  Therefore proposition \ref{Pform} gives
\[
P = -\Om + \Om_\t + 2 H H_\alpha.
\]
Since $g_0(H_\alpha,H_\alpha) = g_0(\alpha,\alpha) = -1/2$, we have $\Om_\t = -2H_\alpha^2$ and equation (\ref{admap}) implies that $iH_\alpha$ corresponds to the element $-\frac{1}{4} (E_\alpha^2 + F_\alpha^2)$.  Now since the sections of the spinor bundle are $T$-invariant functions $SU(2) \to \C$, the element $H_\alpha$ acts as
\[
-m_*H_\alpha = i \sigma ( -\frac{1}{4} (E_\alpha^2 + F_\alpha^2)) = -\frac{1}{2} H.
\]
Thus
\[
P = -\Om - 2 (-\frac{1}{2} H)^2 + 2 H (-\frac{1}{2} H) = -\Om - \frac{3}{2} H^2,
\]
as desired. \\

Now the commutation relations follow from the fact that $\Om$ commutes with all differential operators and proposition \ref{symdol}:
\[
[P, \D] = -\frac{3}{2} [H^2,\D]
\]
and
\begin{gather*}
[H^2, \D] = H^2 \D - \D H^2 = H(\D H + [H, \D]) - \D H^2 \\
= H\D H + H\D - \D H^2 = (\D H + [H, \D]) H + \D H + [H,\D] - \D H^2 \\
= \D H^2 + \D H + \D H + \D - \D H^2 = 2\D H + \D.
\end{gather*}
Similar computations give the formula for $[P,\bar\D]$.
\end{proof}

We now compute the spectrum of $P \vert_{E_l}$.
\begin{prop}\label{specsu2}
The spectrum of $P\vert_{E_l}$ is
\[
\left\{ \lambda_{l,j} :=\frac{1}{8}(4(l+j+1)^2 - 3(2l+1)^2 - 1) : j=0, 1, \ldots \right\}
\]
and the $\lambda_{l,j}$ eigenspace is isomorphic, as an $SU(2)$ representation, to $V_{2(l+j)+1}$.  In particular, the multiplicity of the $\lambda_{l,j}$ eigenspace is $2(l+j+1)$.
\end{prop}
\begin{proof}
By proposition \ref{Psu2}, it is sufficient to determine the spectrum of $\Om\vert_{E_l}$ since $H\vert_{E_l} = -(l+1/2)$.  View sections of $E_l$ as maps $f: SU(2) \to \C$ with $f(ge^h) = e^{-2l-1} f(g)$ and decompose (using the Peter-Weyl theorem) a section $f$ as $f = \sum_{\gamma = 0}^\infty f_\gamma$, where $f_\gamma \in V_\gamma \otimes V_\gamma$.  This decomposition is diagonal with respect to $\Om$, which by proposition \ref{su2cas} acts as $-\frac{(\gamma+1)^2 - 1}{8}$ on $f_\gamma$. \\

The equivariance condition means that for $f_\gamma$ to be non-zero, the second $V_\gamma$ factor of $f_\gamma$ must have $2l+1$ as a weight.  In this case $\gamma$ must be of the form $2(l+j) + 1$, for $j = 0, 1, \ldots$.  Thus $V_\gamma \otimes \{v\}$, $v \in V_\gamma(2(l+j)+1)$, is the eigenspace for $\Om$ with eigenvalue $-\frac{1}{8}(4(l+j+1)^2 - 1)$.
\end{proof}

\begin{prop}\label{su2ladder}
Let $\Gamma_{l,j} \subset \Gamma(E_l)$ denote the $\lambda_{l,j}$ eigenspace of $P\vert_{E_l}$.  Then
\begin{gather*}
\D : \Gamma_{l,j} \to \Gamma_{l-1, j+1} \\
\bar\D : \Gamma_{l,j} \to \Gamma_{l+1,j-1}.
\end{gather*}
Furthermore, these maps are isomorphisms except for the trivial cases of $l=0$ for the first map and $j=0$ for the second.  Then for $N = 0, 1, \ldots,$ the space
\[
\Gamma_N := \bigoplus_{l+j = N} \Gamma_{l,j}
\]
is a $2(N+1)^2$ dimensional representation of the algebra generated by $\{P, H, D, \tilde D\}$.
\end{prop}
\begin{proof}
Let $\psi \in \Gamma_{l,j}$.  Then from proposition \ref{Psu2}
\begin{gather*}
P \D \psi = \D P\psi + [P,\D]\psi = \lambda_{l,j} \D\psi - 3\D H \psi -\frac{3}{2} \D \psi \\
= (\lambda_{l,j} + 3l) \D\psi = \lambda_{l-1, j +1} \D\psi,
\end{gather*}
the last equality coming from a straightforward computation.  Similarly,
\[
P\bar\D \psi = (\lambda_{l,j} - 3l - 3) \bar\D\psi = \lambda_{l+1,j-1}\bar \D\psi.
\]
If $l\ne 0$ then $\D \vert_{\Gamma_{l,j}}$ is injective by proposition (\ref{kersu2}).  Since $\dim \Gamma_{l,j} = 2(l+j+1) =  \dim\Gamma_{l-1,j+1}$, $\D$ must be an isomorphism for $l \ne 0$.   For $\bar\D$, we clearly have $\Gamma_{l,0} \subset \ker \bar\D\vert_{E_l}$.  But from proposition (\ref{kersu2}), $\dim \ker\bar\D\vert_{E_l} = \dim V_{2l+1} = 2l+2 = \dim \Gamma_{l,0}$.  Thus $\bar \D$ is injective, and therefore an isomorphism, as a map $\Gamma_{l,j} \to \Gamma_{l+1,j-1}$ for $j \ne 0$.  Lastly, $\dim \Gamma_N = 2(N+1)^2$ since there are $N+1$-summands, each of which has dimension $2(N+1)$ according to proposition \ref{specsu2}.
\end{proof}

\section{Outlook}
It is our hope that some of the special structure observed in the case of $\C P^1$ may be generalized in some way.  The proof we gave of proposition \ref{su2ladder} is based on what seems to be computation coincidences (e.g. that $\lambda_{l,j} + 3l = \lambda_{l-1,j+1}$).  It would be interesting to find a more conceptual reason for this result and see if it generalizes in some way to higher dimensions.  Also interesting in the $\C P^1$ case is the apparent symmetry of the $l$ and $j$ indices, though the operators whose spectrum they describe are very different ($P$ is second order while $H$ is 0th order).  Propositions \ref{Pform} and \ref{Pspec} give hope that there may be interesting applications to representation theory, in the spirit of Dirac cohomology \cite{dircoh}.

\bibliographystyle{plain}
\bibliography{bibliography}

\end{document}